\documentclass[11pt,a4paper]{amsart}
\usepackage{amssymb,amsmath,amsthm}

\usepackage[utf8]{inputenc}
\usepackage{amsmath}
\usepackage{amssymb}
\usepackage{graphicx}
\usepackage{setspace}
\usepackage{centernot}
\usepackage{amsthm}
\usepackage{fullpage}
\usepackage{amsmath,amssymb,amsfonts}
\usepackage{float}
\usepackage{hyperref}

\newtheorem{theorem}{Theorem}
\newtheorem{lemma}{Lemma}

\newtheorem{definition}{Definition}

\newtheorem{conjecture}{Conjecture}

\DeclareMathOperator{\vol}{Vol}
\newcommand{\Z}{\mathbb Z}
\newcommand{\R}{\mathbb R}

\newcommand{\ext}{\mathrm{ext}}
\newcommand{\dd}{\mathrm{d}}

\makeatletter
\@namedef{subjclassname@2020}{%
  \textup{2020} Mathematics Subject Classification}
\makeatother

\title{The layer number of grids}

\author{Gergely Ambrus}
\address{
Gergely Ambrus\\
Alfr\'ed R\'enyi Institute of Mathematics, Budapest, Hungary}
\email[G. Ambrus]{\texttt ambrus@renyi.hu}

\author{Alexander Hsu}
\address{
Alexander Hsu\\
Department of Mathematics, Reed College, Portland, OR, USA}
\email[A. Hsu]{\texttt hsua@reed.edu}

\author{Bo Peng}
\address{
Bo Peng\\
Department of Mathematics, Carleton College, Northfield, MN, USA}
\email[B. Peng]{\texttt pengg@carleton.edu}

\author{Shiyu Yan}
\address{
Shiyu Yan\\
Department of Mathematics, Carleton College, Northfield, MN, USA}
\email[S. Yan]{\texttt yanj@carleton.edu}

\begin{document}

\begin{abstract} The peeling process is defined as follows: starting with a finite point set $X \subset \R^d$, we repeatedly remove the set of vertices of the convex hull of the current set of points. The number of peeling steps needed to completely delete the set $X$ is called the {\em layer number} of $X$. In this paper, we study the layer number of the $d$-dimensional integer grid $[n]^d$. We prove that for every $d \geq 1$, the layer number of $[n]^d$ is at least $\Omega\left(n^\frac{2d}{d+1}\right)$. On the other hand, we show that for every $d\geq 3$, it takes at most $O(n^{d - 9/11})$ steps to fully remove $[n]^d$. Our approach is based on an enhancement of the method used by Har-Peled and Lidick\'{y} \cite{HL13} for solving the 2-dimensional case.
\end{abstract}

\thanks{Research of the first author was supported by NKFIH grants PD-125502 and KKP-133819.  }



\maketitle

\section{Introduction}

\subsection{History}
Consider  a finite point set $X \subset \R^d$ with $ d \geq 1$. Define the {\em peeling process} as follows: in every step, we remove the set of vertices of the convex hull of the previous iteration.
The sets of points removed  in each step form the {\em convex layers} of $X$, while the total number of steps needed to completely delete $X$ is the {\em layer number} of $X$, which we denote by $\tau(X)$.

The convex layer decomposition of planar sets was first studied by Eddy~\cite{E82} and Chazelle~\cite{C85} from the algorithmic point of view. The latter article gave an $O(n \log n)$ running time algorithm for computing the convex layers of an $n$-element planar point set. Therefore, layer numbers may be computed quickly and efficiently.

Almost 20 years later, Dalal~\cite{D04} determined the expected layer number of random point sets. He proved that if $X$ is a set of $n$ random points chosen independently from the $d$-dimensional unit ball, then $\mathbb{E}(\tau(X)) = \Theta(n^{2/(d+1)})$.

Let $[n]^d=\{1,\ldots,n\}^d$ be the $n\times n\times\ldots\times n$ $d$-dimensional integer grid.  Har-Peled and Lidick{\' y}~\cite{HL13} studied the peeling process of the planar set $[n]^2$. They proved the asymptotically sharp bound $\tau([n]^2)=\Theta(n^{4/3})$, which provides an example when  random points and lattice points behave similarly (see the survey article of B\'{a}r\'{a}ny~\cite{B08} for such phenomena). It is natural to believe that this analogy also holds for higher dimensional cases.

\begin{conjecture}\label{conj1}
The layer number of the grid $[n]^d$ satisfies $\tau([n]^d) = \Theta(n^{2 d/(d+1)})$ for every $d \geq 1$.
\end{conjecture}

If true, the above asymptotic estimate would match the result of Dalal~\cite{D04} on random point sets.  Conjecture~\ref{conj1} initiated our current research project. We study the layer number of the higher dimensional grids $[n]^d$, with particular focus on the 3-dimensional case.  Although we are not able to reach the asymptotically sharp estimate of Conjecture~\ref{conj1}, our results provide the first non-trivial estimates for the layer number of $[n]^d$ with $d\geq 3$.

 Har-Peled and Lidick{\' y} \cite{HL13} also noted that the convex layers of the planar $n \times n$ grid seem to converge to circles. This observation has been given an experimental verification by Eppstein, Har-Peled, and Nivasch \cite{EHN20}, who established an interesting connection between the planar grid peeling process and the affine curve-shortening flow. Further algorithmic applications of the peeling process were given in \cite{ABCR15}.

\subsection{Definitions}
First, we rigorously define the peeling process. Starting with a finite point set $X= X_0$ of $\R^d$, let us recursively define $X_i = X_{i-1} \setminus \ext(X_{i-1})$ for each $i \geq 1$, where $\ext(Y)$ stands for the set of extreme points of $Y$ (that is, the set of vertices of the convex hull of $Y$). The smallest index $i$ for which $X_i = \emptyset$ is called the {\em layer number of $X$}, denoted by $\tau(X)$ (this is sometimes also referred to as the {\em convex depth of $X$}). The point sets removed in each step are called the {\em convex layers} of $X$.

In the article, we are going to study the peeling process of $[n]^d$. For a given $d$ and $n$, set $X = X_0$ to be $[n]^d$, and denote by $P_i$ the convex hull of $X_i$ introduced above. Then $P_i$ is a convex lattice polytope, which is going to be referred to as the {\em $i$th convex layer polytope of $[n]^d$}. Naturally, the polytopes $P_i$ form a nested sequence, starting from the cube $[1,n]^d$, and  shrinking to the empty set.
We also note that since the initial set $[n]^d$ is symmetric, each $P_i$ is symmetric as well.

\medskip

The subsequent arguments are based on lattice geometric observations. A core notion is the following.

\begin{definition}[Primitive vector]
An integer vector $(x_1,x_2,\ldots,x_d) \in \Z^d$ is \textbf{primitive} if not all coordinates are~$0$ and all coordinates are coprime, meaning the greatest common divisor of all $x_i$ is $1$. Geometrically, this translates to the condition that the segment connecting the origin and the primitive point does not contain any other integer points.
\end{definition}
For each positive integer $\mu$, define $V_\mu$ to be the set of primitive vectors with each coordinate contained in the interval $[0, \mu]$.

\begin{definition}[Direction of category k]\label{def2}
Fix a step of the peeling process of $[n]^d$ with the corresponding convex layer polytope $P$. Assume that $P$ is non-degenerate.  For each vector $v \in V_\mu$, there exist two supporting hyperplanes to $P$ that are orthogonal to $v$. By symmetry, these two hyperplanes have isomorphic intersections with $P$. If these intersections are $k$-dimensional faces of $P$, where  $0 \leq k \leq d-1$, we define $v$ to be of \textbf{category} $k$ in this given peeling step.
\end{definition}

By the standard notation, for a convex polytope $P \subset \R^d$, let $f_k (P)$ denote the number of $k$-dimensional faces of $P$, for each $k \in [0,d]$.
For all other common definitions regarding convex sets we refer to the monograph of Schneider~\cite{Sch93}.

All asymptotic notations in the article are meant for a fixed $d$ while $n$ converges to $\infty$, with the implied constants depending on $d$.

\section{Bounds on the layer number}
\subsection{A lower bound on $\tau([n]^d)$}
\begin{theorem}
The layer number of a $d$-dimensional grid $[n]^d$ is bounded below by
$\Omega\left(n^\frac{2d}{d+1}\right)$.
\end{theorem}
\begin{proof}
G.E. Andrews \cite{Andrews} proved that for any convex lattice polytope $P$,
\[
f_0(P) \leq O\left( \vol(P)^\frac{d-1}{d+1} \right),
\]
and this bound is sharp.

Consider the peeling process of $[n]^d$ with the corresponding convex layer polytopes $P_i$. For each~$i$,
the volume of $P_i$ is at most $n^d$. Thus, the upper bound on the number of vertices of each layer is~$(n^d)^\frac{d-1}{d+1}$. Hence, the number of vertices removed in each peeling step is at most $O\big(n^\frac{d^2-d}{d+1}\big)$, which yields the lower bound on the layer number
\[
\tau\left([n]^d\right) \geq \frac{n^d}{O\left(n^\frac{d^2-d}{d+1}\right)}=\Omega\left(n^\frac{2d}{d+1}\right).
\qedhere
\]
\end{proof}

\noindent
Conjecture~\ref{conj1} states that this lower bound is tight.

\subsection{Upper bounds on $\tau([n]^3)$}

The subsequent arguments are based on the approach of Har-Peled and Lidick\'{y} \cite{HL13} for studying the planar case of Conjecture~\ref{conj1}.

\begin{lemma}
Let $H$ be a hyperplane determined by $d$ affinely independent points of $[n]^d$. There exists a primitive vector, with each coordinate bounded by $O(n^{d-1})$, normal to $H$.
\end{lemma}

\begin{proof}
Let the $d$ lattice points  be denoted by $w_1, \ldots, w_d$. Consider the vectors $\mathbf{v}_i = w_i - w_d$ for each $1 \leq i \leq d-1$.  We may find a vector $u$ normal to $H$ by computing the generalized cross product of the vectors $\mathbf{v}_1, \ldots, \mathbf{v}_{d-1}$:
     \[{ u =\displaystyle \bigwedge (\mathbf {v} _{1},\dots ,\mathbf {v} _{d-1})={\begin{vmatrix}v_{1}^1&\cdots &v_{1}^d\\\vdots &\ddots &\vdots \\v_{d-1}^1&\cdots &v_{d-1}^d\\\mathbf {e} _{1}&\cdots &\mathbf {e} _{d}\end{vmatrix}}}\]
where $(\mathbf {e}_1, \ldots, \mathbf {e}_d)$ is the standard basis of $\R^d$, and each $\mathbf{v}_i$ is of the form $\mathbf{v}_i = (v_i^1, \ldots, v_i^d)$. The above formula implies that the coordinates of $u$ are all integers. Let $\tilde{u}$ be the unique primitive vector contained in the segment $\overline{0\, u}$. Then $\tilde{u}$ is also normal to $H$. In order to estimate the coordinates of $\tilde{u}$, we calculate the above determinant through Laplace expansion. The coefficient of each $\mathbf{e}_i$ is the determinant of the $(d-1) \times (d-1)$ matrix obtained by deleting the last row and $i$th column, whose absolute value is at most $(d-1) n^{d-1}$.  Thus, the norm of $u$ is bounded above by $\sqrt{d}(d-1) n^{d-1}=O(n^{d-1})$.
\end{proof}

For any nonzero vector $v \in \Z^d$, let $H_v$ denote the set of hyperplanes orthogonal to $v$ which contain at least one point of $[n]^d$.
\begin{lemma} \label{lem2}
Given a primitive vector $v \in V_\mu$, the number of hyperplanes with normal vector $v$ that intersect $[n]^d$ is bounded by $|H_v|\leq dn\mu$.
\end{lemma}

\begin{proof}
Fix a primitive vector $v = (a_1,a_2,...,a_d) \in V_{\mu}$. Consider a hyperplane $H$ orthogonal to $v$ with its defining equation $a_1x_1+...+a_dx_d=c$. Assume that $H$ contains a point $(x_1, \ldots, x_d) \in [n]^d$.
Since $x_i \in \mathbb{Z}$ and $1 \leq x_i \leq n$ for every $i$, and $0 \leq a_i \leq \mu$ for every $i$,  the value of c is bounded from above by $d\, n  \max(a_i) \leq dn\mu$, while being strictly positive. Since $c$ is an integer, this implies the desired bound.
\end{proof}

\begin{lemma} \label{lem3}
For every $d\geq 3$, $|V_m|= \Theta(m^d)$ holds with the implied constant depending on $d$.
\end{lemma}

\begin{proof}
Avoiding any conflicts with the rest of the paper, for this proof exclusively, let $\mu$ denote the M{\"o}bius function defined in \cite[pp. 234]{HW65}. Introduce Jordan's totient function $J_r(k)$ which counts the $r$-tuples of positive integers all less than or equal to $k$ that form a coprime $(r+1)$-tuple together with $k$. This is a generalization of Euler's totient function, which is given by $J_1$.

Results from Andrica and Piticari \cite{AP03} give us  \[J_r(k)=k^r\sum_{a|k}\frac{\mu(a)}{a^r}=\sum_{a|k} \left(\frac{k}{a}\right)^r\mu(a)=\sum_{aa'=k} (a')^r\mu(a).\]
Now,
\begin{align*}
|V_m|&=\sum_{i=1}^m J_{d-1}(i) = \sum_{aa'\leq m} (a')^{d-1}\mu(a)\\
&=\sum_{a=1}^m \mu(a)\sum_{a'=1}^{\lfloor m/a \rfloor} (a')^{d-1}\\
&=\sum_{a=1}^m \mu(a)\left(\frac{\lfloor m/a \rfloor^d}{d}+\frac{1}{2}
\left\lfloor \frac{m}{a} \right \rfloor^{d-1}+\sum_{i=2}^{d-1}\frac{B_i}{i!}p^{\,\underline{i-1}}\left\lfloor \frac{m}{a} \right \rfloor^{(d-1)-i+1}\right) = (*)
\end{align*}
by Faulhaber's Formula, where $B_i$ denotes the $i$th  Bernoulli number, and $p^{\,\underline{i-1}}=\frac{p!}{(p-i+1)!}$\,. Continuing the above chain of equalities,
\begin{align*}
(*)&=\frac{1}{d}\sum_{a=1}^m \mu(a)\left( \left \lfloor \frac{m}{a} \right \rfloor^d+O \left ( \left \lfloor \frac{m}{a} \right \rfloor^{d-1} \right ) \right )\\
&\geq \frac{1}{d \, 2^d}\sum_{a=1}^m \mu(a)\left( \left(\frac{m}
{a}\right )^d+O\left(\left(\frac{m}{a}\right)^{d-1}\right)\right) \\
&=\frac{m^d}{d \, 2^d} \left(\sum_{a=1}^m \frac{\mu(a)}{a^d} \right) + O\left(m^{d-1}\sum_{a=1}^\infty\frac{1}{a^{d-1}}\right)
\\
&=\frac{m^d}{d \, 2^d} \left(\sum_{a=1}^\infty \frac{\mu(a)}{a^d} \right) -  \frac{m^d}{d \, 2^d} \sum_{a=m+1}^\infty \frac{\mu(a)}{a^d} + O\left(m^{d-1}\right) \\
& \geq \frac{m^d}{d \, 2^d \, \zeta(d)}  - \frac{m^d}{d \, 2^d} \int_{m}^\infty \frac {1}{x^d} \, \dd x + O\left(m^{d-1}\right)
\\
&=\frac{m^d}{d \, 2^d \, \zeta(d)} -O(m)+O\left( m^{d-1} \right)\\
&=\frac{m^d}{d\zeta(d)} +O(m^{d-1}).
\qedhere
\end{align*}

\end{proof}

After the necessary preparations,  we are now able to prove a nontrivial upper bound on the layer number of higher dimensional grids. Here comes our first estimate for the $d=3$ case.

\begin{theorem}\label{thm2}
The number of steps needed to peel away $[n]^3$ is at most $O(n^{9/4})$.
\end{theorem}

\begin{proof}
Consider a given step of the peeling process of $[n]^3$ with the corresponding convex layer polyhedron $P$. Each vector $v \in V_\mu$ may be of category $0,1,$ or $2$, as introduced in Definition~\ref{def2}.

If $v$ is of category $0$, then the intersection of $P$ with either of its supporting planes of normal vector $v$ is just one vertex, which is deleted in the next step of the peeling process. Thus, the number of planes orthogonal to $v$ intersecting the remaining set decreases by at least 2.  By Lemma~\ref{lem2}, initially there are at most $3n\mu$ such planes, hence $v$ may be of category $0$ in at most $\frac 3 2 n\mu$ steps.

If $v$ is of category 1, we know that the intersection of $P$ with its supporting planes orthogonal to $v$ are edges of $P$. Being contained in $[n]^3$, these may not contain more than $n$ grid points. In each subsequent step, the two endpoints of the remaining edge is deleted, so it takes at most $(n+1)/2$ steps to remove the entire edge which is the  intersection of $P$ with its supporting plane. Therefore, by Lemma~\ref{lem2} again, any given $v$ may be of category $1$ in at most $\frac 3 4 n (n+1) \mu$ peeling steps.

Set $M=2n^2\mu$. Since the above arguments show that $v$ may be of category $0$ or $1$ in at most $n^2\mu$ peeling steps, $v$ must be of category $2$ in at least $M/2$ iterations among the first $M$ iterations of the peeling process.

Denote by $c_{k,i}$ the number of category $k$ directions in $V_\mu$ in the $i$th peeling step.

By Euler's polyhedron formula, $f_0(P_i)-f_1(P_i)+f_2(P_i)=2$ holds. Moreover, since $f_1(P_i) \geq 3/2 \, f_2(P_i)$, we also have $f_0(P_i) \geq f_2(P_i)/2$. Since for each category 2 primitive vector $v$ contained in $V_\mu$, there exists a pair of opposites facets of $P$ perpendicular to $v$, $f_2(P_i) \geq 2c_{2,i}$ and hence $f_0(P_i) \geq c_{2,i}$. By Lemma~\ref{lem3}, we reach the following chain of inequalities:
    \begin{equation} \label{eq1}
        \sum_{i=1}^{M} f_0(P_i) \geq \sum_{i=1}^{M} c_{2,i} \geq n^2\mu|V_\mu| \geq \gamma n^2\mu^4
    \end{equation}
with some positive constant $\gamma$.
Setting $\mu=\gamma^{-1/4}\,n^{1/4}$, we find that the total number of vertices of the convex hulls in the first $M$ iterations is at least $n^3$, which is the number of all grid points in $[n]^3$. This implies that the peeling process must terminate in at most  $M=2n^2\mu=O(n^{9/4})$ iterations.
\end{proof}

It should be noted that the above argument may not be applied to the higher dimensional cases, since the number of facets may be much more than the number of vertices (see the relevant Upper Bound Theorem by McMullen~\cite{M70}).

The main restricting factor on our bound is the worst case for the number of steps a primitive vector can be of category $1$, in that it is rare that some, if any, edges will contain $n$ vertices and take $n/2$ steps to fully remove. One expects that as the peeling process evolves, the convex layer polytope has typically only short edges (and, moreover, it converges to a ball). That would lead to an $O(n \mu)$ upper bound for the total number of steps in which a given direction may be of category~1. If that was true, $M$ could be set to be $\Theta(n \mu)$ in the above proof, leading to the desired tight upper bound $\Theta(n^{3/2})$ for the layer number $\tau([n]^3)$. This sketches a possible line of attack for Conjecture~\ref{conj1}.

Although we could not reach the upper bound of Conjecture~\ref{conj1}, we are still able to improve on the result of Theorem~\ref{thm2}. In the remaining part of the section, we present this strengthened bound. The main tool is the following general lemma.

\begin{lemma}\label{lemma4}
    For all $d\geq2$ there exists a positive constant $\alpha_d>0$ such that the number of primitive vectors in $V_\mu$ which are not perpendicular to any non-zero lattice vector of norm at most $\alpha_d  \mu^{1/d}$ is at least $\frac{1}{2}|V_\mu|$.
\end{lemma}
\begin{proof}
    For a given $\nu>0$ to be specified later, we are to count the primitive vectors $v \in V_\mu$ such that the hyperplane $v^\perp$ does not contain any non-zero lattice point of $\Z^d$ with norm less than $\nu$.

    The set of these vectors in $V_\mu$ may be obtained by going over all short non-zero vectors, and dropping out all primitive vectors in $V_\mu$ that are perpendicular to it. For a given $w\in \Z^d$ with $|w|<\nu$, we claim that $w^\perp$ contains at most $\mu^{d-1}$ vectors in $V_\mu$. This follows from selecting a basis vector $e_i$ not contained in $w^\perp$,  and considering the projection of $w^\perp \cap V_\mu$ to $e_i^\perp$ along $e_i$. The projection mapping restricted to $w^\perp$ is one-to-one, it maps lattice points to lattice points, and the image is contained in the $(d-1)$-dimensional cube of edge length $\mu$. Therefore, the number of image points, and thus, the number of points in $w^\perp \cap V_\mu$,  is at most~$\mu^{d-1}$.

   Hence, for each short vector, the number of vectors deleted from  $V_\mu$ in the above process is at most  $\mu^{d-1}$. On the other hand, the number of vectors $w \in \Z^d$ with $|w|\leq \nu$ is at most $(2 \nu + 1)^d = \Theta(\nu^d)$. Thus, the total number of vectors deleted from $V_\mu$ is at most $O(\mu^{d-1} \nu^d)$. Since, by Lemma~\ref{lem3}, $|V_\mu| = \Theta(\mu^d)$, setting $\nu = \alpha_d \mu^{1/d}$ with an appropriate constant $\alpha_d$ (depending on $d$) guarantees that at most half of the vectors of $V_\mu$ are dropped out.
\end{proof}

Utilizing Lemma~\ref{lemma4} we are able to improve our upper bound on the layer number of the 3-dimensional grid.

\begin{theorem} \label{thm3}
    The layer number of $[n]^3$ is at most $O(n^{24/11})$.
\end{theorem}

\begin{proof}
    Let $V'_\mu$ be the set of vectors in $V_\mu$ specified in Lemma \ref{lemma4}. Select $v\in V'_\mu$ arbitrary. Then, any line $\ell$ perpendicular to $v$ may contain at most $s= \sqrt{3}n / (\alpha_3  \mu^{1/3}) + 1$ lattice points of $[n]^3$: the length of the segment $\ell \cap [0,n]^3$ is at most $\sqrt{3} n$, while the shortest nonzero lattice vector parallel to $\ell$ must be of norm at least~$\alpha_3  \mu^{1/3}$.

    Therefore, if $v$ is of category 1 in a given step of the peeling process with the corresponding convex layer polytope $P$, it takes at most $(s+1)/2$ steps to completely remove the edge which is the intersection $P$ and its supporting hyperplane perpendicular to $v$. Thus, by Lemma~\ref{lem2}, the number of steps in which $v$ is of category 1 is at most $N = 6 / \alpha_3 \cdot n^2 \mu^{2/3} $. The number of steps in which $v$ is of category 0 is at most $\frac 3 2 n \mu$.

    Set $M' = 2 (N + \frac 3 2 n \mu)$, and apply the same argument as in the proof of Theorem~\ref{thm2} for the first $M'$ steps of the peeling process, but replacing the set  $V_\mu$ of considered normal vectors by $V'_\mu$. Let $c'_{2,i}$ denote the number of category 2 directions in $V'_\mu$ in the $i$th step of the peeling process of $[n]^3$. Similarly to \eqref{eq1}, we arrive at the following chain of inequalities:
    \begin{equation*}
        \sum_{i=1}^{M'} f_0(P_i) \geq \sum_{i=1}^{M'} c'_{2,i} \geq \frac{M'}{2}|V'_\mu| = \Theta(n^2\mu^{2/3})\Theta(\mu^3) = \gamma' n^2\mu^{11/3}
    \end{equation*}
    with  some  positive  constant $\gamma'$. We finish the proof by setting $\mu=(\gamma')^{-3/11}\,n^{3/11}$, which by the same argument as before results in the upper bound $\tau([n]^3) \leq O(n^{24/11})$.
\end{proof}

\subsection{Upper bound for higher dimensional cases}
We conclude the paper by the extension of the bound of  Theorem~\ref{thm3} to higher dimensions using a simple recursive argument.

\begin{theorem}
For every $d \geq 3$, the layer number of the $d$-dimensional grid satisfies $\tau([n]^d) \leq O(n^{d-9/11})$.
\end{theorem}

\begin{proof} We will apply induction on $d$. The $d=3$ case is provided by Theorem~\ref{thm3}. We shall prove the $d$-dimensional case, assuming that the estimate is valid for $[n]^{d-1}$.

Consider any set $A \subset [n]^d$. Let $H$ be the hyperplane defined by the equation $x_1 = 1$, which is tangent to the cube $[1,n]^d$. The restriction of the ($d$-dimensional) peeling process of $A$ to $H$ is then equivalent to the $(d-1)$-dimensional peeling process of $A \cap H$. Since $A \cap H$ is contained in a copy of $[n]^{d-1}$, this must terminate in at most $O(n^{(d-1) - 9/11})$ steps, by the inductive hypothesis. The same reasoning may be applied to each of the $2d$ boundary hyperplanes of the cube $[1,n]^d$. Thus, we obtain that after $O(n^{(d-1) - 9/11})$ peeling steps of $A$, the remaining set will contain no points on the boundary of the cube $[1,n]^d$.

Now, we note that the grid $[n]^d$ may be written as the union of $\lceil n/2 \rceil$ cubic shells. Applying the above argument for each of these shells results in the upper bound $O(n^{d - 9/11})$ for the layer number of $[n]^d$.
\end{proof}

\section{Acknowledgements}

\noindent
This research was done under the auspices of the
Budapest Semesters in Mathematics program.

\end{document}